\documentclass[a4paper, 11pt]{article}

\usepackage[top=1in, bottom=1.1in, left=0.9in, right=0.8in]{geometry}
\usepackage[parfill]{parskip}
\usepackage{amsmath}
\usepackage{amsthm}
\usepackage{graphicx}
\usepackage{amssymb}
\usepackage{epstopdf}
\usepackage{setspace}
\usepackage{authblk}
\usepackage{enumerate}
\usepackage{lmodern}
\usepackage[hypcap]{caption}
\usepackage{color}
\definecolor{dblue}{rgb}{0,0,0.45}
\definecolor{red}{rgb}{0.7,0,0}

\usepackage[english]{babel}
\usepackage{fancyhdr}

\newtheorem{theorem}{Theorem}[section]
\newtheorem{lemma}[theorem]{Lemma}
\newtheorem{proposition}[theorem]{Proposition}

\theoremstyle{definition}

\newtheorem{example}[theorem]{Example}
\theoremstyle{remark}

\numberwithin{equation}{section}

\def\N{\mathbb{N}}
\def\w{\mathbb{\omega}}
\def\Z{\mathbb{Z}}
\def\R{\mathbb{R}}

\begin{document}

\title{Discrete Morrey Spaces and Their Inclusion Properties}
\author{Hendra Gunawan$^1$, Eder Kikianty$^2$, and Christopher Schwanke$^3$}
\affil{$^1$Department of Mathematics, Bandung Institute of Technology\\
Jalan Ganesha No. 10, Bandung 40132, Indonesia\\E-mail:~{\rm hgunawan@math.itb.ac.id}}
\affil{$^2$Department of Mathematics and Applied Mathematics, University of Pretoria\\
Private Bag X20, Hatfield 0028, South Africa\\E-mail:~{\rm eder.kikianty@up.ac.za}}
\affil{$^3$Unit for BMI, North-West University\\Private Bag X6001, Potchefstroom 2520,
South Africa\\E-mail:~{\rm schwankc326@gmail.com}}
\date{}

\maketitle

\begin{abstract}
We discuss discrete Morrey spaces and their generalizations, and we prove
necessary and sufficient conditions for the inclusion property among
these spaces through an estimate for the characteristic sequences.

\bigskip

\noindent{\bf Keywords}: Discrete Morrey spaces, inclusion properties

\noindent{\bf Mathematics Subject Classification}: 42B35, 46B45, 46A45
\end{abstract}

\section{Introduction}

Many operators that are initially studied on Lebesgue spaces $L^p(\mathbb{R}^d)$
have discrete analogues on $\ell^p(\mathbb{Z}^d)$, see for examples
\cite{Kov, MSW, Ob, SW1, SW2, SW3, SW4}.
Some of these operators have also been studied on `continuous' Morrey spaces
$M^p_q(\mathbb{R}^d)$, see for examples \cite{Adams, Chiarenza, Gunawan, Nakai,
Sawano}. In this paper, we are interested in studying discrete analogues of
Morrey spaces and their generalizations. In particular, we discuss the inclusion
property of these spaces and prove some necessary and sufficient conditions for
this property. For related works on the continuous version, see \cite{GHLM, Os, Raf}.

Let $m\in\mathbb{Z}$, $N\in \omega:=\mathbb{N}\cup\{0\}$, and write $S_{m,N}:=
\{m-N, \dots, m,\dots, m+N\}$. Then $|S_{m,N}|=2N+1$ --- the cardinality of $S_{m,N}$.
(It is independent of $m$, but we write it in the subscript to keep track of the set.)
Let $\mathbb{K}$ be $\mathbb{R}$ or $\mathbb{C}$ and $1\leq p\leq q<\infty$.
We denote by $\ell^p_q=\ell^p_q(\mathbb{Z})$ the set of sequences $x=(x_k)_{k\in\mathbb{Z}}$
taking values in $\mathbb{K}$ such that
$$\|x\|_{\ell^p_q}:=\sup_{m\in\mathbb{Z},N\in\omega} |S_{m,N}|^{\frac1q-\frac1p}
\left(\sum_{k\in S_{m,N}} |x_k|^p\right)^\frac1p<\infty.$$
Clearly $\ell^p_q$ is a vector space, which we shall call a {\it discrete Morrey space}.
We remark that when $p=q$, we have $\ell^p_p = \ell^p$, the space of $p$-summable
sequences with integer indices. For a sequence $x$ to be in $\ell^p_q$, $x$ has
to have some decay, but not as fast as those in $\ell^p$. In general, for $p<q$,
$\ell^p_q$ is a larger space than $\ell^p$, as we shall see below.

In the following sections, we shall discuss (i) discrete Morrey spaces, (ii) weak
type discrete Morrey spaces, (iii) generalized discrete Morrey spaces, and
(iv) generalized weak type discrete Morrey spaces. In each section, we prove the
inclusion relations between the spaces. In particular, in the last two sections, we prove
necessary and sufficient conditions for the inclusion property among the generalized discrete
Morrey spaces and also among the generalized weak type discrete Morrey spaces. Our main
results are presented in Theorems \ref{strongequivalence} and \ref{weakequivalence}.

\section{Discrete Morrey Spaces}

We begin with the following proposition, which tells us that the discrete Morrey space
$\ell^p_q$ contains $\ell^p$, for $p\le q$.

\bigskip

\begin{proposition}
For $1\leq p\leq q <\infty$, we have $\ell^p \subseteq \ell^p_q$ and $\| x\|_{\ell^p_q}
\leq \|x\|_{\ell^p}$ for every $x\in \ell^p$. For $1\le p<q<\infty$, the inclusion is strict.
\end{proposition}

\begin{proof}
Let $x=(x_k)_{k\in\Z}\in\ell^p$. We have for all $m\in\mathbb{Z}$ and $N\in \omega$,
$0< |S_{m,N}|^{\frac1q-\frac1p}\leq 1$, and thus
$$|S_{m,N}|^{\frac1q-\frac1p} \left(\sum_{k\in S_{m,N}} |x_k|^p\right)^\frac1p \leq
\left(\sum_{k\in S_{m,N}} |x_k|^p\right)^\frac1p.$$
Taking the supremum over $m\in\mathbb{Z}$ and $N\in \omega$, we get $\| x\|_{\ell^p_q}
\leq \|x\|_{\ell^p}.$

To show that the inclusion is strict for $1\leq p<q<\infty$, consider the sequence
$x=(x_k)_{k\in \mathbb{Z}}$ given
by $x_k:=|k|^{-\frac1q}$ when $k \neq 0$ and $x_0:=1$. Since $\frac{p}{q}<1$, the series
$$\sum_{k\in \mathbb{Z}} |x_k|^p=1+2\sum_{k=1}^\infty k^{-\frac{p}{q}}$$
is divergent, thus $x \not\in \ell^p(\mathbb{Z})$. Next, for all $m\in\mathbb{Z}$
and $N\in \omega$, we have
$$\sum_{k\in S_{m,N}} |x_k|^p \le \sum_{k\in S_{0,N+1}} |x_k|^p
= 1 + 2 \sum_{k=1}^{N+1} k^{-\frac{p}{q}}.$$
Using the lower Riemann sum of $\int\limits_{1}^{N+1} {x^{-\frac{p}{q}}}\,dx$, we see that
$$
\sum_{k=1}^{N+1} k^{-\frac{p}{q}}\le 1+\int_{1}^{N+1} x^{-\frac{p}{q}} dx=
1-\frac{q}{q-p}+\frac{q}{q-p}(N+1)^{1-\frac pq}.
$$
Since $(2N+1)^{\frac{p}{q}-1}\le 1$ and $N+1\le 2N+1$, we have
\[
|S_{m,N}|^{\frac pq-1}\sum_{k\in S_{m,N}} |x_k|^p \leq (2N+1)^{\frac{p}{q}-1}
\left(3-\frac{2q}{q-p}+\frac{2q}{q-p}(N+1)^{1-\frac pq}\right)
\le 3+\frac{2q}{q-p}.
\]
Taking the $p$-th roots and then the supremum over all $m\in \mathbb{Z}$ and $N\in\omega$,
we obtain
\[
\sup_{m\in\mathbb{Z},N\in\omega} |S_{m,N}|^{\frac 1q-\frac1p}\left(\sum_{k\in S_{m,N}}
|x_k|^p \right)^{\frac1p}\le \left(3+\frac{2q}{q-p}\right)^\frac1p,
\]
and thus $x\in \ell^p_q$.
\end{proof}

\bigskip

\begin{proposition}\label{prop:2-2}
For $1\leq p\leq q <\infty$, the mapping $\|\cdot\|_{\ell^p_q}$ defines a
norm on ${\ell^p_q}$. Moreover, $(\ell^p_q,\|\cdot\|_{\ell^p_q})$ is a Banach
space.
\end{proposition}

\begin{proof}
It is easy to see that $\|x\|_{\ell^p_q} \geq 0$ for every $x \in \ell^p_q$, and
that $\|x\|_{\ell^p_q}=0$ iff $x=0$.
Next, we have $\|\alpha x\|_{\ell^p_q}=|\alpha|\|x\|_{\ell^p_q}$ for every $x\in \ell^p_q$
and $\alpha \in \mathbb{K}$.
Now, let $x=(x_k)_{k\in\Z}, y=(y_k)_{k\in\Z} \in \ell^p_q$, $m\in\mathbb{Z}$, and $N\in\omega$.
By using Minkowski's inequality, we have
\[
|S_{m,N}|^{\frac1q-\frac1p} \left(\sum_{k\in S_{m,N}} |x_k+y_k|^p\right)^\frac1p
\leq |S_{m,N}|^{\frac1q-\frac1p} \left(\sum_{k\in S_{m,N}} |x_k|^p\right)^\frac1p+
|S_{m,N}|^{\frac1q-\frac1p} \left(\sum_{k\in S_{m,N}} |y_k|^p\right)^\frac1p.
\]
Taking the supremum over $m\in\mathbb{Z}$ and $N\in \omega$, we get $\|x+y\|_{\ell^p_q}
\leq \|x\|_{\ell^p_q}+\|y\|_{\ell^p_q}$. All these show that $\|\cdot\|_{\ell^p_q}$ defines
a norm on $\ell^p_q$.

We shall now show that $(\ell^p_q,\|\cdot\|_{\ell^p_q})$ is a Banach space.
Let $\varepsilon >0$ and $(x^{(n)})_{n\in\mathbb{N}}$ be a Cauchy sequence in $\ell^p_q$.
Then there exists $n_\varepsilon \in \omega$ such that
\begin{equation}\label{eq:cauchy-seq}
\sup_{m\in\mathbb{Z},N\in\omega} |S_{m,N}|^{\frac1q-\frac1p} \left(\sum_{k\in S_{m,N}}
|x^{(i)}_k-x^{(j)}_k|^p \right)^\frac1p<\varepsilon,
\end{equation}
for $i,j \geq n_{\varepsilon}$.
Consequently, for every $m\in\mathbb{Z}$ and $N\in \omega$, we have
$$|S_{m,N}|^{\frac1q-\frac1p}\left(\sum_{k\in S_{m,N}} \big|x^{(i)}_k-x^{(j)}_k\big|^p
\right)^\frac1p <\varepsilon,
$$
for $i,j \geq n_\varepsilon$. By taking $N=0$, we obtain for each $k\in\mathbb{Z}$
$$
\big|x^{(i)}_k-x^{(j)}_k\big|<\varepsilon,
$$
for $i,j \geq n_\varepsilon$. Thus $(x^{(n)}_k)_{n\in\mathbb{N}}$ is a Cauchy sequence in $\mathbb{K}$
for each $k\in \mathbb{Z}$. Define $x:=(x_k)_{k\in\mathbb{Z}}$ where
$$x_k :=\lim_{n\rightarrow \infty} x^{(n)}_k,\quad k\in\mathbb{Z}.$$
If we let $j\rightarrow \infty$ in \eqref{eq:cauchy-seq}, then we have
$$\sup_{m\in\mathbb{Z},N\in\omega} |S_{m,N}|^{\frac1q-\frac1p} \left(\sum_{k\in S_{m,N}}
|x^{(i)}_k-x_k|^p\right)^\frac1p <\varepsilon,
$$
for $i \geq n_\varepsilon$. Hence $x=x^{(i)}-(x^{(i)}-x)$ is in $\ell^p_q$ and the above result shows that
$x^{(i)} \to x$ as $i \rightarrow \infty$.
\end{proof}

The following lemma will be useful in studying the relation between two discrete Morrey
spaces.

\bigskip

\begin{lemma}\label{lemma:monotonicity}
For all $1\leq p_1\leq p_2 < \infty$, $m\in\mathbb{Z}$, and $N\in \omega$, we have
$$
 \left(\frac1{|S_{m,N}|}\sum_{k\in S_{m,N}} |x_k|^{p_1}\right)^\frac1{p_1} \leq
 \left(\frac1{|S_{m,N}|}\sum_{k\in S_{m,N}} |x_k|^{p_2}\right)^\frac1{p_2},
$$
where $x_k\in \mathbb{K}$ for all $k\in S_{m,N}$.
\end{lemma}

\begin{proof}
Let $1\leq p_1\leq p_2 < \infty$, $m\in\mathbb{Z}$, and $N\in \omega$.
By H\"older's inequality, we have
\begin{align*}
\sum_{k\in S_{m,N}} |x_k|^{p_1} &\leq \left(\sum_{k\in S_{m,N}} |x_k|^{p_2}
\right)^\frac{p_1}{p_2} \left(\sum_{k\in S_{m,N}} 1\right)^{1-\frac{p_1}{p_2}}\\
&= |S_{m,N}|^{1-\frac{p_1}{p_2}} \left(\sum_{k\in S_{m,N}} |x_k|^{p_2}
\right)^\frac{p_1}{p_2}\\
&= |S_{m,N}|\left(\frac{1}{|S_{m,N}|}\sum_{k\in S_{m,N}}
|x_k|^{p_2}\right)^\frac{p_1}{p_2}.
\end{align*}
Thus
$$
\frac{1}{|S_{m,N}|}\sum_{k\in S_{m,N}} |x_k|^{p_1} \leq \left(\frac{1}{|S_{m,N}|}
\sum_{k\in S_{m,N}} |x_k|^{p_2}\right)^\frac{p_1}{p_2},
$$
and this completes the proof.
\end{proof}

\bigskip

\begin{proposition}\label{prop:inclusion1}
For all $1\leq p_1\leq p_2 \leq q< \infty$, we have $\ell^{p_2}_q \subseteq
\ell^{p_1}_q$ with $\|x\|_{\ell^{p_1}_q} \le \|x\|_{\ell^{p_2}_q}$ for every
$x\in \ell^{p_2}_q$.
\end{proposition}

\begin{proof}
The proof follows immediately from Lemma \ref{lemma:monotonicity}.
\end{proof}

\bigskip

\noindent{\tt Remark.} At the present we do not know whether the inclusion in Proposition
\ref{prop:inclusion1} is strict, as we do not have an example of a sequence
which is in $\ell^{p_1}_q$ but not in $\ell^{p_2}_q$ for $1\leq p_1< p_2<q<\infty$.

\section{Weak Type Discrete Morrey spaces}

For $1\leq p\leq q <\infty$, we define the {\it weak type discrete Morrey space}
$w\ell^p_q$ to be the set of sequences $x=(x_k)_{k\in\mathbb{Z}}$ taking values in
$\mathbb{K}$ such that $\|x\|_{w\ell^p_q}<\infty$, where $\|\cdot\|_{w\ell^p_q}$ given by
$$\|x\|_{w\ell^p_q } := \sup_{m\in\mathbb{Z}, N \in\omega, \gamma>0 }
|S_{m,N}|^{\frac1q-\frac1p}\gamma \big|\{k\in S_{m,N}:\ |x_k| >\gamma\}\big|^\frac1p.$$
Note that when $p=q$, we have $w\ell^p := w\ell^p_p$, which is a weak type of
$\ell^p$ space. The following example shows that $w\ell^p_p$ has more elements that $\ell^p$.

\bigskip

\begin{example}
The sequence $x=(x_k)_{k\in \mathbb{Z}}$ given by $x_k:=|k|^{-\frac1p}$ when $k \neq 0$
and $x_0:=1$ is not in $\ell^p$. Nevertheless, for any $m\in\mathbb{Z}, N\in\omega$, and
$0<\gamma<1$, we have
\begin{align*}
\gamma |\{k\in S_{m,N}: |x_k|>\gamma\}|^\frac1p
&\le \gamma |\{k\in S_{0,N}: |x_k|>\gamma\}|^\frac1p\\
&\le \gamma \Bigl(1 + 2|\{k\in\mathbb{N} :
1\le k\leq N,\ k^{-\frac1p}>\gamma\}|^\frac1p\Bigr) \\
&< \gamma \Bigl(1 + \frac{2}{\gamma}\Bigr)<3.
\end{align*}
Thus $(x_k)_{k\in \mathbb{Z}}$ is in $w\ell^p_p$.
\end{example}

\bigskip

\begin{theorem}\label{thm:weak-contains-strong}
For $1\leq p\leq q <\infty$, $\ell^p_q \subseteq w\ell^p_q$ with $\|x\|_{w\ell^p_q}
\le \|x\|_{\ell^p_q}$ for every $x\in \ell^p_q$.
\end{theorem}

\begin{proof}
Let $x\in \ell^p_q$, $m\in\mathbb{Z}$, $N\in \omega$, and $\gamma>0$. We have
\begin{align*}
|S_{m,N}|^{\frac1q-\frac1p} \gamma \big|\{k\in S_{m,N}:\ |x_k| >\gamma\}\big|^\frac1p
&=   |S_{m,N}|^{\frac1q-\frac1p}\bigg( \sum_{k\in S_{m,N},|x_k| >\gamma} \gamma^p
\bigg)^\frac1p \\
&\leq   |S_{m,N}|^{\frac1q-\frac1p}\bigg( \sum_{k\in S_{m,N}, |x_k| >\gamma} |x_k|^p
\bigg)^\frac1p \\
&\leq   |S_{m,N}|^{\frac1q-\frac1p}\bigg( \sum_{k\in S_{m,N}} |x_k|^p\bigg)^\frac1p.
\end{align*}
Taking the supremum over $m\in\mathbb{Z}$, $N\in \omega$, and $\gamma>0$, we obtain
$\|x\|_{w\ell^p_q}\leq\|x\|_{\ell^p_q}$. Therefore, if $x \in \ell^p_q$, then $x \in
w\ell^p_q$.
\end{proof}


\bigskip

\begin{theorem}
For $1\leq p\leq q <\infty$, $\|\cdot\|_{w\ell^p_q } $ is a quasi-norm, so that
$(w\ell^p_q,\|\cdot\|_{w\ell^p_q})$ is a quasi-normed space.
\end{theorem}

\begin{proof}
From the definition of $\|\cdot\|_{w\ell^p_q }$, it is clear that
$\|x\|_{w\ell^p_q } \geq 0$ for all $x \in w\ell^p_q $. Let $m\in\mathbb{Z}$,
$N\in \omega$, and $\gamma>0$. If $x=0$, then $\{k\in S_{m,N}:\ |x_k| >\gamma\}$ is
an empty set, and therefore its cardinality is zero. Thus
$$ |S_{m,N}|^{\frac1q-\frac1p}\gamma \big|\{k\in S_{m,N}:\ |x_k| >
\gamma\}\big|^\frac1p=0.$$
Taking the supremum over $m\in\mathbb{Z}$, $N\in \omega$, and $\gamma>0$,
we obtain $\|x\|_{w\ell^p_q }= 0$. Conversely, suppose that $\|x\|_{w\ell^p_q } = 0$. Then
$$ \big|\{k\in S_{m,N}:\ |x_k| >\gamma\}\big|=0$$
for all $m\in\mathbb{Z},\ N\in \omega$, and $\gamma >0$.
We conclude that for all $k\in\mathbb{Z}$, we have $0 \leq |x_k|  \leq \gamma$ for
all $\gamma >0$. Thus $x_k=0$ for all $k\in \mathbb{Z}$, that is, $x=0$.

Next let $x=(x_k)_{k\in\Z} \in {w\ell^p_q }$ and $\alpha \in \mathbb{K}$. When $\alpha=0$,
clearly $\|\alpha x\|_{w\ell^p_q } =|\alpha| \|x\|_{w\ell^p_q }$. Suppose $\alpha
\neq 0$. We have
\begin{align*}
\|\alpha x\|_{w\ell^p_q } &=\sup_{m\in\mathbb{Z}, N \in\omega, \gamma>0 }
|S_{m,N}|^{\frac1q-\frac1p} \gamma \big|\{k\in S_{m,N}:\ |\alpha x_k| >\gamma\}\big|^\frac1p\\
&=\sup_{m\in\mathbb{Z}, N \in\omega, \gamma>0 } |S_{m,N}|^{\frac1q-\frac1p} \gamma
\left|\left\{k\in S_{m,N}:\ |x_k| >\frac{\gamma}{|\alpha|}\right\}\right|^\frac1p\\
&=\sup_{m\in\mathbb{Z}, N \in\omega, \delta>0 } |S_{m,N}|^{\frac1q-\frac1p} \delta
|\alpha|\big|\{k\in S_{m,N}:\ |x_k| >\delta\}\big|^\frac1p=|\alpha|\| x\|_{w\ell^p_q }.
\end{align*}
Now let $ x=(x_k)_{k\in\Z},y=(y_k)_{k\in\Z} \in {w\ell^p_q }$. For any $m\in\mathbb{Z},\
N\in\omega$, and $\gamma>0$, we observe that
\begin{align*}
\{k\in S_{m,N}: |x_k+y_k| >\gamma\}
&\subseteq \{k\in S_{m,N}: |x_k|+|y_k| >\gamma\} \\
&\subseteq \{k\in S_{m,N}: |x_k| >\frac{\gamma}{2}\} \cup \{k\in S_{m,N}: |y_k| >\frac{\gamma}{2}\}.
\end{align*}
Hence for any $m\in\mathbb{Z}$, $N\in \omega$, and $\gamma>0$, we have
\begin{align*}
&\left(|S_{m,N}|^{\frac1q-\frac1p}\gamma\right)^p  \big|\{k\in S_{m,N}:\ |x_k+y_k| >\gamma\}\big|\\
&\quad\leq \left(|S_{m,N}|^{\frac1q-\frac1p} \gamma\right)^p \big|\{k\in S_{m,N}:\ |x_k| >\frac{\gamma}{2}\}\big|
+ \left(|S_{m,N}|^{\frac1q-\frac1p}\gamma\right)^p  \big|\{k\in S_{m,N}:\ |y_k| >\frac{\gamma}{2}\}\big|.
\end{align*}
By writing $\delta=\gamma/2$ on the right hand side of the above inequality, we get
\begin{align*}
&\left(|S_{m,N}|^{\frac1q-\frac1p} \gamma \right)^p\big|\{k\in S_{m,N}:\ |x_k+y_k| >\gamma\}\big|\\
&\quad\leq 2^p\left(|S_{m,N}|^{\frac1q-\frac1p} \delta \right)^p\big|\{k\in S_{m,N}:\ |x_k| >\delta\}\big|
+ 2^p\left(|S_{m,N}|^{\frac1q-\frac1p} \delta \right)^p  \big|\{k\in S_{m,N}:\ |y_k| >\delta\}\big|,
\end{align*}
whence
\begin{align*}
&|S_{m,N}|^{\frac1q-\frac1p} \gamma \big|\{k\in S_{m,N}:\ |x_k+y_k| >\gamma\}\big|^\frac1p\\
&\quad\leq 2\left[\left(|S_{m,N}|^{\frac1q-\frac1p}\delta\right)^p  \big|\{k\in S_{m,N}:
\ |x_k| >\delta\}\big|
+ \left(|S_{m,N}|^{\frac1q-\frac1p}\delta\right)^p  \big|\{k\in S_{m,N}:\ |y_k| >\delta\}\big|\right]^\frac1p\\
&\quad\leq 2\left[|S_{m,N}|^{\frac1q-\frac1p} \delta \big|\{k\in S_{m,N}:\ |x_k| >\delta\}\big|^\frac1p\right]
 + 2\left[|S_{m,N}|^{\frac1q-\frac1p} \delta \big|\{k\in S_{m,N}:\ |y_k| >\delta\}\big|^\frac1p\right]\\
&\quad\leq 2(\|x\|_{w\ell^p_q }+\|y\|_{w\ell^p_q }).
\end{align*}
Taking the supremum over $m\in\mathbb{Z}$, $N\in \omega$, and $\gamma>0$, we get $\|x+y\|_{w\ell^p_q }\leq
2(\|x\|_{w\ell^p_q }+\|y\|_{w\ell^p_q }).$
\end{proof}

\bigskip

\begin{proposition}\label{weak-completeness}
For $1\leq p\leq q<\infty$, $w\ell_{q}^{p}$ is complete with respect to the quasi-norm $\|\cdot\|_{w\ell_{q}^{p}}$.
\end{proposition}

\begin{proof}
As before, we denote sequences in $w\ell_{q}^{p}$ by $(x^{(n)})_{n\in\N}$, where $x^{(n)}=(x_k^{(n)})_{k\in\Z}$
for each $n\in\N$.
Let $1\leq p\leq q<\infty$, and let $(x^{(n)})_{n\in\N}$ be a Cauchy sequence in $w\ell_{q}^{p}$. We first
note that $(x^{(n)})_{n\in\N}$ must be bounded in $w\ell_{q}^{p}$.
Next we show that there exists a sequence $(x_{k})_{k\in\Z}$ taking values in $\mathbb{K}$ such that
$x_{k}^{(n)}\rightarrow x_{k}$ as $n\to\infty$, for each $k\in\Z$.
To do so, let $0<\epsilon\leq 1$. Choose $M\in\N$ such that for every $k_{0}\in\Z$ and all $i,j\geq M$,
\[
|S_{k_{0},0}|^{\frac{1}{q}-\frac{1}{p}}\epsilon|\{k\in S_{k_{0},0}:|x^{(i)}_{k}-x^{(j)}_{k}|>\epsilon\}|^{1/p}
\leq\|x^{(i)}-x^{(j)}\|_{w\ell_{q}^{p}}<\epsilon^{\frac{1}{p}+1}.
\]
Note that $S_{k_{0},0}=\{k_{0}\}$ for every $k_{0}\in\Z$. It thus follows that
\[
|\{k\in S_{k_{0},0}:|x^{(i)}_{k}-x^{(j)}_{k}|>\epsilon\}|<\epsilon.
\]
for every $k_{0}\in\Z$ and $i,j\geq M$. This implies that
\[
|x^{(i)}_{k_{0}}-x^{(j)}_{k_{0}}|\leq\epsilon
\]
for every $k_{0}\in\Z$ and $i,j\geq M$.
Hence, for each $k\in\Z$, $(x_{k}^{(n)})_{n\in\N}$ is a Cauchy sequence taking values in $\mathbb{K}$.
Therefore, $\lim\limits_{n\rightarrow\infty}x_k^{(n)}$ exists for each $k\in\Z$. Now for each $k\in\Z$, define
$x_{k}:=\lim\limits_{n\rightarrow\infty}x_{k}^{(n)}$, so that $x_{k}^{(n)}\rightarrow x_{k}$ as $n\to\infty$.

We claim that $x:=(x_k)_{k\in\Z}\in w\ell_q^p$. Let $C>0$ satisfy $\|x^{(n)}\|_{w\ell_{q}^{p}}\leq C\ (n\in\N)$,
and let $m\in\Z,\ N\in\w$, and $\gamma>0$. We observe that for any $n\in\N$,
\begin{align*}
&|S_{m,N}|^{\frac{1}{q}-\frac{1}{p}}\frac{\gamma}{2}|\{k\in S_{m,N}:|x_{k}|>\gamma\}|^{\frac{1}{p}}\\
&\leq|S_{m,N}|^{\frac{1}{q}-\frac{1}{p}}\frac{\gamma}{2}\left(\left|\left\{k\in S_{m,N}:|x_{k}-x^{(n)}_{k}|>
\frac{\gamma}{2}\right\}\right|
+\left|\left\{k\in S_{m,N}:|x^{(n)}_{k}|>\frac{\gamma}{2}\right\}\right|\right)^{\frac{1}{p}}\\
&\leq|S_{m,N}|^{\frac{1}{q}-\frac{1}{p}}\frac{\gamma}{2}\left|\left\{k\in S_{m,N}:|x_{k}-x^{(n)}_{k}|>
\frac{\gamma}{2}\right\}\right|^{\frac{1}{p}}
+|S_{m,N}|^{\frac{1}{q}-\frac{1}{p}}\frac{\gamma}{2}\left|\left\{k\in S_{m,N}:|x^{(n)}_{k}|>
\frac{\gamma}{2}\right\}\right|^{\frac{1}{p}}.
\end{align*}
Since $x_{k}^{(n)}\rightarrow x_{k}$ as $n\to\infty$ for each $k\in\Z$ and $S_{m,N}$ is finite,
we can choose $n_{0}\in\N$
large enough so that $|x^{(n_{0})}_{k}-x_{k}|\leq\frac{\gamma}{2}$ for $k\in S_{m,N}$. Then
\[
|S_{m,N}|^{\frac{1}{q}-\frac{1}{p}}\frac{\gamma}{2}\left|\left\{k\in S_{m,N}:|x_{k}-x^{(n_{0})}_{k}|>
\frac{\gamma}{2}\right\}\right|^{\frac{1}{p}}=0.
\]
It follows that
\begin{align*}
|S_{m,N}|^{\frac{1}{q}-\frac{1}{p}}\frac{\gamma}{2}|\{k\in S_{m,N}:|x_{k}|>\gamma\}|^{\frac{1}{p}}&\leq
|S_{m,N}|^{\frac{1}{q}-\frac{1}{p}}\frac{\gamma}{2}\left|\left\{k\in S_{m,N}:|x^{(n_{0})}_{k}|>
\frac{\gamma}{2}\right\}\right|^{\frac{1}{p}}\\
&\leq\|x_{k}^{(n_{0})}\|_{w\ell_{q}^{p}}\leq C.
\end{align*}
Taking the supremum over all $m\in\Z,\ N\in\w$, and $\gamma>0$, we have
$\frac{1}{2}||x||_{w\ell_{q}^{p}}\leq C<\infty$
and thus $\|x\|_{w\ell_{q}^{p}}<\infty$. Therefore, $x\in w\ell_{q}^{p}$, as claimed.

Finally, we show that $x^{(n)}\to x$ as $n\to\infty$ in the quasi norm $\|\cdot\|_{w\ell_q^p}$.
For this, let $\epsilon>0$.
Fix $j\in\N,\ m\in\Z,\ N\in\w$, and $\gamma>0$. Since for $k\in\Z$ we have $|x_{k}-x_{k}^{(j)}|=
\lim\limits_{i\rightarrow\infty}|x_{k}^{(i)}-x_{k}^{(j)}|$, we obtain $|x_{k}-x_{k}^{(j)}|>\gamma$
if and only if there exists $M\in\N$ such that $|x_{k}^{(i)}-x_{k}^{(j)}|>\gamma$ for $i\geq M$.
Using this fact, one readily proves the identity
\[
\{k\in S_{m,N}:|x_{k}-x_{k}^{(j)}|>\gamma\}=\bigcup_{M=1}^{\infty}\bigcap_{i=M}^{\infty}
\{k\in S_{m,N}:|x_{k}^{(i)}-x_{k}^{(j)}|>\gamma\}.
\]
Using the continuity of counting measure, we obtain
\begin{align*}
|S_{m,N}|^{\frac{1}{q}-\frac{1}{p}}\gamma &|\{k\in S_{m,N}:|x_{k}-x^{(j)}_{k}|>\gamma\}|^{\frac{1}{p}}\\
&=|S_{m,N}|^{\frac{1}{q}-\frac{1}{p}}\gamma\left|\bigcup_{M=1}^{\infty}\bigcap_{i=M}^{\infty}
\{k\in S_{m,N}:|x^{(i)}_{k}-x^{(j)}_{k}|>\gamma\}\right|^{\frac{1}{p}}\\
&=\lim_{M\rightarrow\infty}|S_{m,N}|^{\frac{1}{q}-\frac{1}{p}}\gamma\left|\bigcap_{i=M}^{\infty}
\{k\in S_{m,N}:|x^{(i)}_{k}-x^{(j)}_{k}|>\gamma\}\right|^{\frac{1}{p}}.
\end{align*}
For any $M\in\N$,
\begin{align*}
|S_{m,N}|^{\frac{1}{q}-\frac{1}{p}}\gamma &\left|\bigcap_{i=M}^{\infty}\{k\in S_{m,N}:
|x^{(i)}_{k}-x^{(j)}_{k}|>\gamma\}\right|^{\frac{1}{p}}\\
&\leq|S_{m,N}|^{\frac{1}{q}-\frac{1}{p}}\gamma\left|\{k\in S_{m,N}:
|x^{(M)}_{k}-x^{(j)}_{k}|>\gamma\}\right|^{\frac{1}{p}}\\
&\leq||x^{(M)}-x^{(j)}||_{w\ell_{q}^{p}}.
\end{align*}
Since $(x^{(n)})_{n\in\N}$ is Cauchy in $w\ell_{q}^{p}$, there exists
$P\in\N$ (which does not depend on $m,N$, or $\gamma$) such that if $M,j\geq P$, then
\[
\|x^{(M)}-x^{(j)}\|_{w\ell_{q}^{p}}<\epsilon.
\]
Thus for $j\geq P$,
\[
\lim_{M\rightarrow\infty}|S_{m,N}|^{\frac{1}{q}-\frac{1}{p}}\gamma
\left|\bigcap_{i=M}^{\infty}\{k\in S_{m,N}:|x^{(i)}_{k}-x^{(j)}_{k}|>\gamma\}\right|^{\frac{1}{p}}<\epsilon.
\]
Summarizing, we have for $j\geq P$,
\[
|S_{m,N}|^{\frac{1}{q}-\frac{1}{p}}\gamma|\{k\in S_{m,N}:|x_{k}-x^{(j)}_{k}|>\gamma\}|^{\frac{1}{p}}<\epsilon.
\]
Taking the supremum over all $m\in\Z,\ N\in\w$, and $\gamma>0$, we obtain for $j\geq P$,
\[
\|x-x^{(j)}\|_{w\ell_{q}^{p}}<\epsilon.
\]
Hence $x^{(n)}\rightarrow x$ as $n\to\infty$ in $w\ell_{q}^{p}$, and this ends the proof.
\end{proof}

\medskip

The next proposition gives the inclusion property between two weak type discrete
Morrey spaces.

\bigskip

\begin{proposition}
Let $1\leq p_1\leq p_2\leq q <\infty$. Then $w\ell^{p_2}_q \subseteq w\ell^{p_1}_q$ with
$\|x\|_{w\ell^{p_1}_q} \le \|x\|_{w\ell^{p_2}_q}$ for every $x\in w\ell^{p_2}_q$.
\end{proposition}

\begin{proof}
Let $x \in w\ell^{p_2}_q$ and $\gamma >0$. By definition, we have
$$ |S_{m,N}|^{\frac1q-\frac1{p_2}} \gamma \big|\{k\in S_{m,N}:\ |x_k| >\gamma\}\big|^\frac1{p_2}
\leq \|x\|_{w\ell^{p_2}_q },$$
for any $m\in\mathbb{Z}$ and $N\in \w$. Assuming that $|\{k\in S_{m,N}:\ |x_k| >\gamma\}\big|\not=0$,
we have
$$\gamma \leq \frac{ |S_{m,N}|^{\frac1{p_2}-\frac1q} } {\big|\{k\in S_{m,N}:\ |x_k| >\gamma\}
\big|^\frac1{p_2}} \|x\|_{w\ell^{p_2}_q }.$$
Therefore, for any $m\in\mathbb{Z}$ and $N\in \w$ , we have
\begin{align*}
|S_{m,N}|^{\frac1q-\frac1{p_1}}\gamma \big|\{k\in S_{m,N}:\ |x_k| >\gamma\}\big|^\frac1{p_1}
&\quad\leq \frac{ |S_{m,N}|^{\frac1{p_2}-\frac1{p_1}} } {\big|\{k\in S_{m,N}:\ |x_k| >\gamma\}
\big|^{\frac1{p_2}-\frac1{p_1}}} \|x\|_{w\ell^{p_2}_q }\\
&\quad= \left(\frac{\big|\{k\in S_{m,N}:\ |x_k| >\gamma\}\big|}{|S_{m,N}| }\right)^{\frac1{p_1}-
\frac1{p_2}} \|x\|_{w\ell^{p_2}_q }\\
&\quad\leq \|x\|_{w\ell^{p_2}_q }.
\end{align*}
We see that the inequality also holds when $|\{k\in S_{m,N}:\ |x_k| >\gamma\}\big|=0$.
Taking the supremum over $m\in\mathbb{Z}$, $N\in \w$, and $\gamma>0$, we obtain
$ \|x\|_{w\ell^{p_1}_q }\leq  \|x\|_{w\ell^{p_2}_q }$, and the proof is complete.
\end{proof}

\section{Generalized Discrete Morrey Spaces}

The generalized discrete Morrey space $\ell^p_\phi$ is equipped with two parameters,
that is, $1\leq p < \infty$ and a function $\phi \in \mathcal{G}_p$, where
$\mathcal{G}_p=\mathcal{G}_p(2\omega + 1)$ is the set of all functions $\phi:2\omega + 1\rightarrow (0,\infty)$
such that $\phi$ is {\it almost decreasing} (that is, there exists $C>0$ such that $\phi(2M+1) \ge
C\,\phi(2N+1)$ for $M,N\in\omega$ with $M\le N$), and the mapping $(2N+1)\mapsto (2N+1)^\frac{1}{p}\phi(2N+1)$ is
{\it almost increasing} (that is, there exists $C>0$ such that $(2M+1)^\frac1p \phi(2M+1) \le
C\,(2N+1)^\frac1p \phi(2N+1)$ for $M,N\in\omega$ with $M\le N$).
Note that $\phi\in \mathcal{G}_p$ implies that $\phi$ satisfies the {\it doubling condition},
that is, there exists $C>0$ such that
$$\frac1C \leq \frac{\phi(2M+1)}{\phi(2N+1)} \leq C$$
whenever $\frac12\leq \frac{2M+1}{2N+1}\leq 2$.

For $1\leq p <\infty$ and $\phi \in \mathcal{G}_p$, the {\it generalized discrete
Morrey space} $\ell^p_\phi$ is defined as the set of all sequences
$x=(x_k)_{k=1}^\infty$ taking values in $\mathbb{K}$ such that
$$\|x\|_{\ell^p_\phi} := \sup_{m\in\mathbb{Z},N\in\omega} \frac{1}{\phi(2N+1)} \left(\frac{1}{|S_{m,N}|}
\sum_{k\in S_{m,N}} |x_k|^p\right)^\frac1p <\infty.$$
Note that the discrete Morrey space $\ell^p_q$ ($1\leq p\leq q <\infty$) may be obtained
from $\ell^p_\phi$ by choosing the function $\phi(2N+1)=(2N+1)^{-\frac1q},\ N\in\omega$.

The proof of our next proposition is similar to the proof of Proposition \ref{prop:2-2}.

\bigskip

\begin{proposition}
For $1\leq p <\infty$ and $\phi\in \mathcal{G}_p$, the mapping $\|\cdot\|_{\ell^p_\phi}$ defines a
norm on ${\ell^p_\phi}$. Moreover, $(\ell^p_q,\|\cdot\|_{\ell^p_\phi})$ is a Banach space.
\end{proposition}

\bigskip

The following lemma gives an estimate for the norm of the characteristic sequences, which will
be useful later on.

\bigskip

\begin{lemma}\label{lemma:characteristic}
Let $1\leq p < \infty$ and $\phi \in\mathcal{G}_p$. For $m_0\in\mathbb{Z}$ and
$N_0 \in \omega$, let $\xi^{m_0,N_0}$ be the characteristic sequence given by
\begin{equation}\label{eq:characteristic}
\xi^{m_0,N_0}_k :=\left\{
\begin{array}{ll}
1, &\text{if}\ k \in S_{m_0,N_0},\\
0, &\text{otherwise.}
\end{array}\right.
\end{equation}
Then there exists $C>0$, independent of $m_0$ and $N_0$, such that
$$\frac{1}{\phi({2N_0+1})} \leq \|\xi^{m_0,N_0}\|_{\ell^p_\phi} \le \frac{C}{\phi({2N_0+1})}$$
for every $N_0\in\omega$.
\end{lemma}

\begin{proof}
We fix $m_0\in\mathbb{Z}$ and $N_0 \in \omega$. Then we have
\begin{align*}
 \|\xi^{m_0,N_0}\|_{\ell^p_\phi} &= \sup_{m\in\mathbb{Z},N\in\omega} \frac{1}{\phi(2N+1)}
 \left(\frac{1}{|S_{m,N}|}
 \sum_{k\in S_{m,N}} |\xi^{m_0,N_0}_k|^p\right)^\frac1p \\
 &\geq \frac{1}{\phi(2N_0+1)}\left(\frac{|S_{m_0,N_0}|}{|S_{m_0,N_0}|} \right)^\frac1p =
 \frac{1}{\phi(2N_0+1)}.
\end{align*}
For the second inequality, take any $m\in\Z$ and $N\in\omega$. If $N\le N_0$, we use the fact that
$\phi$ is almost decreasing: there exists $C_1>0$ such that $\phi(2N+1)\ge C_1\, \phi(2N_0+1)$,
and that $\sum\limits_{k\in S_{m_0,N}} |\xi^{m_0,N_0}_k|^p = |S_{m_0,N}|.$
Hence
\begin{align*}
\frac{1}{\phi(2N+1)} \left(\frac{1}{|S_{m,N}|} \sum_{k\in S_{m,N}} |\xi^{m_0,N_0}_k|^p\right)^\frac1p
&\le \frac{1}{\phi(2N+1)} \left(\frac{1}{|S_{m_0,N}|} \sum_{k\in S_{m_0,N}} |\xi^{m_0,N_0}_k|^p\right)^\frac1p\\
&= \frac{1}{\phi(2N+1)} \left(\frac{|S_{m_0,N}|}{|S_{m_0,N}|}\right)^\frac1p\le \frac{1}{C_1\phi(2N_0+1)}.
\end{align*}
If $N\ge N_0$, there exists $C_2>0$ such that $(2N_0+1)^\frac1p \phi(2N_0+1) \le C_2\, (2N+1)^\frac1p
\phi(2N+1)$. In this case, we have
\begin{align*}
 \frac{1}{\phi(2N+1)} \left(\frac{1}{|S_{m,N}|} \sum_{k\in S_{m,N}} |\xi^{m_0,N_0}_k|^p\right)^\frac1p
 &\le \frac{C_2(2N+1)^\frac1p }{(2N_0+1)^{\frac1p}\phi(2N_0+1)} \left(\frac{|S_{m_0,N_0}|}{|S_{m_0,N}|}\right)^\frac1p\\
&= \frac{C_2}{\phi(2N_0+1)}.
\end{align*}
The constants $C_1$ and $C_2$ are independent of $m_0$, $m$, $N_0$, and $N$.
Taking the supremum over $m\in\Z$ and $N\in\omega$,
we get $\|\xi^{m_0,N_0}\|_{\ell^p_\phi}\le \frac{C}{\phi(2N_0+1)}$,
where $C=\max\{\frac{1}{C_1},C_2\}$. This completes the proof.
\end{proof}

\medskip

Let $X\not=\emptyset$. For $f,g:X\to \R$, we write $f\lesssim g$ (or $g\gtrsim f$) if
there exists a constant $C>0$ such that $f(x)\le C\,g(x)$ for every $x\in X$.

\bigskip

\begin{theorem}\label{strongequivalence}
Let $1\leq p_1 \leq p_2 <\infty$, $\phi_1 \in \mathcal{G}_{p_1}$, and $\phi_2 \in \mathcal{G}_{p_2}$.
Then the following statements are equivalent:
\begin{enumerate}[(i)]
\item $\phi_2 \lesssim \phi_1$ (on $2\omega+1$).
\item $\|\cdot\|_{\ell^{p_1}_{\phi_1}}\lesssim \|\cdot\|_{\ell^{p_2}_{\phi_2}}$ (on $\ell^{p_2}_{\phi_2}$).
\item $\ell^{p_2}_{\phi_2} \subseteq \ell^{p_1}_{\phi_1}$.
\end{enumerate}
\end{theorem}

\begin{proof}
We first prove that (i) and (ii) are equivalent. Suppose that (i) holds. Let $x \in  \ell^{p_2}_{\phi_2}$.
For any $m\in\mathbb{Z}$ and $N \in \omega$, we have
\begin{align*}
\frac{1}{\phi_1(2N+1)} \left(\frac{1}{|S_{m,N}|} \sum_{k\in S_{m,N}} |x_k|^{p_1}\right)^\frac1{p_1}
&\le \frac{C}{\phi_2(2N+1)}\left(\frac{1}{|S_{m,N}|} \sum_{k\in S_{m,N}} |x_k|^{p_1}\right)^\frac1{p_1}\\
&\le \frac{C}{\phi_2(2N+1)}\left(\frac{1}{|S_{m,N}|} \sum_{k\in S_{m,N}} |x_k|^{p_2}\right)^\frac1{p_2},
\end{align*}
for some $C>0$ (independent of $m$ and $N$). Note the use of Lemma \ref{lemma:monotonicity}
in the last inequality. Taking the supremum over $m\in\mathbb{Z}$ and $N \in \omega$,
we obtain $\|\cdot\|_{\ell^{p_1}_{\phi_1}} \lesssim \|\cdot\|_{\ell^{p_2}_{\phi_2}}$
(on $\ell^{p_2}_{\phi_2}$). Thus (ii) holds.

Next suppose that (ii) holds. Let $m_0\in\mathbb{Z}$, $N_0 \in \omega$, and
$\xi^{m_0,N_0}$ be the characteristic sequence
defined by \eqref{eq:characteristic} in Lemma \ref{lemma:characteristic}. By our assumption,
$\|\xi^{m_0,N_0}\|_{\ell^{p_1}_{\phi_1}} \le C_1\, \|\xi^{m_0,N_0}\|_{\ell^{p_2}_{\phi_2}}$
for some $C_1>0$. Meanwhile, Lemma \ref{lemma:characteristic} gives us
$$\frac{1}{\phi_1(2N_0+1)} \leq \|\xi^{m_0,N_0}\|_{\ell^{p_1}_{\phi_1}}\quad \text{and}
\quad \|\xi^{m_0,N_0}\|_{\ell^{p_2}_{\phi_2}} \le \frac{C_2}{\phi_2(2N_0+1)},$$
for some $C_2>0$. Both $C_1$ and $C_2$ are independent of $m_0$ and $N_0$. We conclude that
$$\frac{1}{\phi_1(2N_0+1)} \le \frac{C_1C_2}{\phi_2(2N_0+1)},\quad \text{or equivalently,}
\quad \phi_2 (2N_0+1)\le C_1C_2\phi_1(2N_0+1);$$
and this tells us that (i) holds since the above inequality holds for any $N_0\in \mathbb{N}$.

We shall now prove that (ii) and (iii) are equivalent. But (ii) clearly implies (iii),
and so it remains only to show that (iii) implies (ii).

For any $x\in\ell_{\phi_{2}}^{p_{2}}$, we have $\|x\|_{\ell_{\phi_{2}}^{p_{2}}}<\infty$
and by assumption we also know that $\|x\|_{\ell_{\phi_{1}}^{p_{1}}}<\infty$.
Define $|||x|||:=\|x\|_{\ell_{\phi_{1}}^{p_{1}}}+\|x\|_{\ell_{\phi_{2}}^{p_{2}}}$
for every $x\in\ell_{\phi_{2}}^{p_{2}}$.
Note that $|||\cdot|||$ is a norm on $\ell_{\phi_{2}}^{p_{2}}$. Moreover, as in \cite{Os},
one may verify that $(\ell_{\phi_{2}}^{p_{2}},|||\cdot|||)$ is a Banach space.

Now consider the identity mapping $I:(\ell_{\phi_{2}}^{p_{2}},|||\cdot|||)\rightarrow
(\ell_{\phi_{2}}^{p_{2}},||\cdot||_{\ell_{\phi_{2}}^{p_{2}}})$. If $x^{(n)}\rightarrow x$
in $(\ell_{\phi_{2}}^{p_{2}},|||\cdot|||)$, then $x^{(n)}\rightarrow x$ in
$(\ell_{\phi_{2}}^{p_{2}},\|\cdot\|_{\ell_{\phi_{2}}^{p_{2}}})$ as well, since
$\|\cdot\|_{\ell_{\phi_{2}}^{p_{2}}}\leq|||\cdot|||$. This tells us that $I$ is a continuous linear operator.
Evidently, $(\ell_{\phi_{2}}^{p_{2}},|||\cdot|||)$ is a closed subspace of
$(\ell_{\phi_{2}}^{p_{2}},\|\cdot\|_{\ell_{\phi_{2}}^{p_{2}}})$. It follows from the
Open Mapping Theorem that $I$ is open. Since $I$ is bijective, we know that $I^{-1}$ is
continuous, and hence bounded. It follows that there exists $C>0$ such that
$|||x|||\leq C\|x\|_{\ell_{\phi_{2}}^{p_{2}}}$ for every $x\in\ell_{\phi_{2}}^{p_{2}}$. Therefore,
\[
\|x\|_{\ell_{\phi_{1}}^{p_{1}}}\leq|||x|||\leq C\|x\|_{\ell_{\phi_{2}}^{p_{2}}}.
\]
for every $x\in\ell_{\phi_{2}}^{p_{2}}$. This completes the proof.
\end{proof}

\section{Generalized Weak Type Discrete Morrey Spaces}

For $1\leq p <\infty$ and $\phi \in \mathcal{G}_p$, the {\it generalized weak type
discrete Morrey space} $w\ell^p_\phi$ is the set of all sequences
$x=(x_k)_{k\in\mathbb{Z}}$ taking values in $\mathbb{K}$ such that
$\|x\|_{w\ell^p_\phi}< \infty$, where $\|\cdot\|_{w\ell^p_\phi}$ is defined by
$$\|x\|_{w\ell^p_\phi } := \sup_{m\in\mathbb{Z}, N \in\omega, \gamma>0 } \frac{\gamma}{\phi(2N+1)}
\left( \frac{\big|\{k\in S_{m,N}:\ |x_k| >\gamma\}\big|}{|S_{m,N}|}\right)^\frac1p.$$

The proof of our next proposition is similar to the proof of Proposition \ref{weak-completeness}.

\bigskip

\begin{proposition}
For $1\leq p <\infty$ and $\phi \in \mathcal{G}_p$, $(w\ell^p_\phi, \|\cdot\|_{w\ell^p_\phi})$
is a quasi-Banach space.
\end{proposition}

\bigskip

\begin{proposition}\label{prop:weak-inclusion-generalised}
For $1\leq p <\infty$ and $\phi \in \mathcal{G}_p$, $\ell^p_\phi \subseteq w\ell^p_\phi$
with $\|x\|_{w\ell^p_\phi} \le \|x\|_{\ell^p_\phi}$ for every $x\in \ell^p_\phi$.
\end{proposition}

\begin{proof}
Let $x\in \ell^p_\phi$, $m\in\mathbb{Z}$, $N\in \omega$, and $\gamma >0$. We have
\begin{align*}
\frac{\gamma}{\phi(2N+1)}\left( \frac{\big|\{k\in S_{m,N}:\ |x_k| >\gamma\}\big|}{|S_{m,N}|}\right)^\frac1p
&= \frac{1}{\phi(2N+1)}\left( \frac{\gamma^p\big|\{k\in S_{m,N}:\ |x_k| >\gamma\}\big|}{|S_{m,N}|}\right)^\frac1p\\
&= \frac{1}{\phi(2N+1)}\left( \frac{1}{|S_{m,N}|}\sum_{k\in S_{m,N},|x_k| >\gamma} \gamma^p\right)^\frac1p\\
&\le \frac{1}{\phi(2N+1)}\left( \frac{1}{|S_{m,N}|}\sum_{k\in S_{m,N},|x_k| >\gamma} |x_k|^p\right)^\frac1p\\
&\le \frac{1}{\phi(2N+1)}\left( \frac{1}{|S_{m,N}|}\sum_{k\in S_{m,N}} |x_k|^p\right)^\frac1p.
\end{align*}
Taking the supremum over $m\in\mathbb{Z}$, $N\in \omega$, and $\gamma>0$, we obtain
$\|x\|_{w\ell^p_\phi} \leq\|x\|_{\ell^p_\phi}$. Therefore, $\ell^p_\phi \subseteq w\ell^p_\phi$.
\end{proof}

\bigskip

\begin{lemma}\label{lemma:characteristic-weak}
Let $1\leq p < \infty$ and $\phi \in\mathcal{G}_p$. If $m_0\in\mathbb{Z}$ and $N_0 \in \omega$, and
$\xi^{m_0,N_0}$ is the characteristic sequence defined by \eqref{eq:characteristic} in Lemma
\ref{lemma:characteristic}, then there exists $C>0$, independent of $m_0$ and $N_0$, such that
$$\frac{1}{2\phi(2N_0+1)} \leq \|\xi^{m_0,N_0}\|_{w\ell^p_\phi} \le \frac{C}{\phi(2N_0+1)}.$$
\end{lemma}

\begin{proof}
Let $m_0\in\mathbb{Z}$ and $N_0\in\omega$. By definition, we have
\begin{align*}
\|\xi^{m_0,N_0}\|_{w\ell^p_\phi} &= \sup_{m\in\mathbb{Z}, N \in\omega, \gamma>0 } \frac{\gamma}{\phi(2N+1)}
\left( \frac{\big|\{k\in S_{m,N}:\ |\xi^{m_0,N_0}_k| >\gamma\}\big|}{|S_{m,N}|}\right)^\frac1p\\
&\geq \frac{\frac12}{\phi(2N_0+1)}\left( \frac{\big|\{k\in S_{m_0,N_0}:\ |\xi_k^{m_0,N_0}| >\frac12\}
\big|}{|S_{m_0,N_0}|}\right)^\frac1p\\
&\geq \frac{1}{2\phi(2N_0+1)}\left( \frac{|S_{m_0,N_0}|}{|S_{m_0,N_0}|}\right)^\frac1p= \frac{1}{2\phi(2N_0+1)}.
\end{align*}
Next, by using Lemma \ref{lemma:characteristic} and Proposition \ref{prop:weak-inclusion-generalised},
there exists $C>0$ independent of $m_0$ and $N_0$ such that
$$\|\xi^{m_0,N_0}\|_{w\ell^p_\phi} \leq \|\xi^{m_0,N_0}\|_{\ell^p_\phi}\le \frac{C}{\phi(2N_0+1)},$$
and this completes the proof.
\end{proof}

\bigskip

\begin{theorem}\label{weakequivalence}
Let $1\leq p_1 \leq p_2 <\infty$, $\phi_1 \in \mathcal{G}_{p_1}$ and $\phi_2 \in
\mathcal{G}_{p_2}$. Then the following statements are equivalent:
\begin{enumerate}[(i)]
\item $\phi_2 \lesssim \phi_1$ (on $2\omega+1$).
\item $\|\cdot\|_{w\ell^{p_1}_{\phi_1}} \lesssim \|\cdot\|_{w\ell^{p_2}_{\phi_2}}$
(on $w\ell^{p_2}_{\phi_2}$).
\item $w\ell^{p_2}_{\phi_2} \subseteq w\ell^{p_1}_{\phi_1}$.
\end{enumerate}
\end{theorem}

\begin{proof}
As before, we can prove first that (i) and (ii) are equivalent. Suppose that (i) holds.
Let $x \in  w\ell^{p_2}_{\phi_2}$, $m\in\mathbb{Z}$,
$N\in \omega$, and $\gamma >0$. By definition, we have
$$\frac{\gamma}{\phi_2(2N+1)}\left( \frac{\big|\{k\in S_{m,N}:\ |x_k| >\gamma\}\big|}
{|S_{m,N}|}\right)^\frac1{p_2}\leq \|x\|_{w\ell^{p_2}_{\phi_2}},$$
or, assuming that $|\{k\in S_{m,N}:\ |x_k|>\gamma\}|\not=0$,
$$\frac{\gamma}{\phi_2(2N+1) } \leq\left( \frac{\big|\{k\in S_{m,N}:\ |x_k| >\gamma\}
\big|}{|S_{m,N}|}\right)^{-\frac1{p_2}} \|x\|_{w\ell^{p_2}_\phi }.$$
Therefore, we obtain
\begin{align*}
\frac{\gamma}{\phi_1(2N+1)}\left( \frac{\big|\{k\in S_{m,N}:\ |x_k| >\gamma\}\big|}
{|S_{m,N}|}\right)^\frac1{p_1} &\le \frac{C\gamma}{\phi_2(2N+1)}\left(
\frac{\big|\{k\in S_{m,N}:\ |x_k| >\gamma\}\big|}{|S_{m,N}|}\right)^\frac1{p_1}\\
&\leq C\left( \frac{\big|\{k\in S_{m,N}:\ |x_k| >\gamma\}\big|}{|S_{m,N}|}
\right)^{\frac1{p_1}-\frac1{p_2}}  \|x\|_{w\ell^{p_2}_\phi }\\
&\leq C\|x\|_{w\ell^{p_2}_\phi },
\end{align*}
for some $C>0$ independent of $m, N$, and $\gamma$. Note that the inequality still holds when
$|\{k\in S_{m,N}:\ |x_k| >\gamma\}\big|=0$. Taking the supremum over
$m\in\mathbb{Z}$, $N\in \omega$, and $\gamma>0$, we obtain
$\|x\|_{w\ell^{p_1}_\phi }\le C\,\|x\|_{w\ell^{p_2}_\phi}$. Thus (ii) holds.

Next suppose that (ii) holds. Let $m_0\in\mathbb{Z}$, $N_0 \in \omega$, and
$\xi^{m_0,N_0}$ be the characteristic sequence defined by \eqref{eq:characteristic}
in Lemma \ref{lemma:characteristic}. Then we have $\|\xi^{m_0,N_0}\|_{w\ell^{p_1}_{\phi_1}} \le
C_1\,\|\xi^{m_0,N_0}\|_{w\ell^{p_2}_{\phi_2}}$ for some $C_1>0$ independent of $m_0$ and $N_0$.
By Lemma \ref{lemma:characteristic-weak}, we have
$$\frac{1}{2\phi_1(2N_0+1)} \leq \|\xi^{m_0,N_0}\|_{w\ell^{p_1}_{\phi_1}}\quad \text{and}
\quad \|\xi^{m_0,N_0}\|_{w\ell^{p_2}_{\phi_2}} \le \frac{C_2}{\phi_2(2N_0+1)}$$
for some $C_2>0$ independent of $m_0$ and $N_0$. We conclude that
$$\frac{1}{\phi_1(2N_0+1)} \le \frac{C}{\phi_2(2N_0+1)},\quad \text{or equivalently,}
\quad \phi_2 (2N_0+1)\le C\,\phi_1(2N_0+1),$$
for some $C>0$ independent of $N_0$. This means that (i) holds, since $N_0\in \omega$ is arbritary.

Since (ii) also implies (iii), it now remains to prove that (iii) implies (ii).
Let $x\in w\ell_{\phi_{2}}^{p_{2}}$. Then $\|x\|_{w\ell_{\phi_{2}}^{p_{2}}}<\infty$,
and by assumption we also have $\|x\|_{w\ell_{\phi_{1}}^{p_{1}}}<\infty$.
Define $|||x|||:=\|x\|_{w\ell_{\phi_{1}}^{p_{1}}}+\|x\|_{w\ell_{\phi_{2}}^{p_{2}}}\
(x\in w\ell_{\phi_{2}}^{p_{2}})$. Note that $|||\cdot|||$ is a quasi-norm on $w\ell_{\phi_{2}}^{p_{2}}$
and one may observe that $(w\ell_{\phi_{2}}^{p_{2}},|||\cdot|||)$ is a quasi-Banach space.

Next, analogous to the proof of Theorem \ref{strongequivalence}, we consider the identity mapping
$I:(w\ell_{\phi_{2}}^{p_{2}},|||\cdot|||)\rightarrow (w\ell_{\phi_{2}}^{p_{2}},\|\cdot\|_{w\ell_{\phi_{2}}^{p_{2}}})$.
If $x^{(n)}\rightarrow x$ in $(w\ell_{\phi_{2}}^{p_{2}},|||\cdot|||)$, then $x^{(n)}\rightarrow x$
in $(w\ell_{\phi_{2}}^{p_{2}},\|\cdot\|_{w\ell_{\phi_{2}}^{p_{2}}})$ as well, since $\|\cdot\|_{w\ell_{\phi_{2}}^{p_{2}}}\leq
|||\cdot|||$. Thus $I$ is a continuous linear operator. Now $(w\ell_{\phi_{2}}^{p_{2}},|||\cdot|||)$
is a closed subspace of $(w\ell_{\phi_{2}}^{p_{2}},\|\cdot\|_{\ell_{\phi_{2}}^{p_{2}}})$.
Since the Open Mapping Theorem also holds for quasi-Banach spaces, it follows that $I$ is open.
As $I$ is bijective, $I^{-1}$ is continuous and hence bounded. Thus there exists $C>0$ such that
$|||x|||\leq C\|x\|_{w\ell_{\phi_{2}}^{p_{2}}}$ for every $x\in w\ell_{\phi_{2}}^{p_{2}}$. Therefore,
\[
\|x\|_{w\ell_{\phi_{1}}^{p_{1}}}\leq|||x|||\leq C\|x\|_{w\ell_{\phi_{2}}^{p_{2}}}
\]
for every $x\in w\ell_{\phi_{2}}^{p_{2}}$, and we are done.
\end{proof}

\medskip

\subsection*{Acknowledgement} The work was initiated during the second author's
visit to Department of Mathematics, Bandung Institute of Technology (Math-ITB) in 2014
and completed during the third author's postdoctoral program at Math-ITB in 2016.
The first author is supported by ITB Research and Innovation Program 2016.

\end{document}